\documentclass{amsart}
\usepackage[dvips]{graphicx}
\usepackage[all]{xypic}
\usepackage{amsmath,graphics}
\usepackage{amsfonts,amssymb}

\theoremstyle{plain}
\newtheorem*{theorem*}{Theorem}
\newtheorem*{lemma*} {Lemma}
\newtheorem*{corollary*} {Corollary}
\newtheorem*{proposition*} {Proposition}
\newtheorem{theorem}{Theorem}[section]
\newtheorem{lemma}[theorem]{Lemma}
\newtheorem{corollary}[theorem]{Corollary}
\newtheorem{proposition}[theorem]{Proposition}

\theoremstyle{remark}
\newtheorem*{remark}{Remark}

\newtheorem*{example}{Example}

\newtheorem*{ack}{Acknowledgement}
\theoremstyle{definition}

\newcommand{\MC}{{\mathcal M}}

\newcommand{\Z}{\mathbb{Z}}
\newcommand{\N}{\mathbb{N}}

\newcommand{\R}{\mathbb{R}}
\newcommand{\C}{\mathbb{C}}
\newcommand{\K}{\mathbb{K}}
\newcommand{\PP}{\mathbb{P}}

\def\ug{{\mathcal{U}(\G)}}
\def\kg{\K(\G)}

\def\ng{{\mathcal{N}(\G)}}

\def\cp{\mathbf{CP}}

\def\R{\Bbb{R}}

\def\K{\Bbb{K}}

\def\id{\mbox{id}}

\def\Z{\Bbb{Z}}
\def\C{\Bbb{C}}

\def\N{\Bbb{N}}

\def\part{\partial}

\def\BB{\mathcal{B}}
\def\g{\gamma}
\def\sm{\setminus}
\def\bn{\begin{enumerate}}

\def\en{\end{enumerate}}
\def\ba{\begin{array}}
\def\ea{\end{array}}

\def\a{\alpha}

\def\ti{\widetilde}

\def\fr12{\frac{1}{2}}

\def\im{\mbox{Im}}

\def\ker{\mbox{Ker}}

\def\be{\begin{equation}}
\def\ee{\end{equation}}

\def\K{\Bbb{K}}

\def\G{\Gamma}

\def\wti{\widetilde}

\def\K{\Bbb{K}}

\def\scra{\mathcal{A}}
\def\cp{\Bbb{CP}}

\begin{document}

\title{$L^2$--Betti numbers of hypersurface complements}
\author{Laurentiu Maxim}
\address{Department of Mathematics,
          University of Wisconsin-Madison,
          480 Lincoln Drive, Madison WI 53706-1388, USA.}
\email {maxim@math.wisc.edu}

\thanks{The author was partially supported by NSF-1005338.}

\date{\today}



\begin{abstract}
In \cite{DJL07} it was shown that if $\scra$ is an affine hyperplane arrangement in $\C^n$, then at most one of the $L^2$--Betti numbers $b_i^{(2)}(\C^n\sm  \scra,\id)$  is non--zero. In this note we prove an analogous statement
for complements of complex affine hypersurfaces in general position at infinity. Furthermore, we recast and extend to this higher-dimensional setting results of \cite{FLM,LM06} about $L^2$--Betti numbers of plane curve complements. 
\end{abstract}

\maketitle

\section{Introduction}
Let $M$ be any  topological space and $\a:\pi_1(M)\to \G$ an epimorphism to a  group $\G$ (all groups are assumed countable).
Then for $i\in \N\cup \{0\}$ we can consider the $L^2$--Betti number $b_i^{(2)}(M,\a)\in [0,\infty]$.
We recall the definition and some of the most important properties of $L^2$--Betti numbers in Section \ref{section:l2betti}.

Let $X\subset \C^{n}$ ($n \geq 2$) be a reduced affine hypersurface defined by a polynomial  equation $f=f_1 \cdots f_s=0$,  where $f_i$ are the irreducible factors of $f$. Denote by 
 $X_i:=\{f_i=0\}$, $i=1,\cdots,s$, the irreducible components of $X$, and let
 $$M_X:=\C^{n}\sm X$$ be the hypersurface complement. Then $M_X$ has the homotopy type of a finite CW complex of dimension $n$.
It is well--known  that $H_1(M_X;\Z)$ is a free abelian group generated by 
the meridian loops $\gamma_i$ about the
non-singular part of each irreducible component $X_i$ of $X$.
Throughout the paper we denote by $\phi$ the map $\pi_1(M_X)\to \Z$ given by sending each meridian $\g_i$ to $1$. This is the same map as the homomorphism $f_*:\pi_1(M_X) \to \pi_1(\C^*)=\Z$ induced by $f$. We also refer to $\phi$ as the {\it  total linking number  homomorphism}.
We call an epimorphism $\a:\pi_1(M_X)\to \G$ \emph{admissible} if the total linking number homomorphism $\phi$ factors through $\a$. 

The main result of this note is the following ``nonresonance-type" theorem.
\begin{theorem} \label{mainthm}
Let $X\subset \C^{n}$  be a reduced affine hypersurface in general position at infinity, i.e., whose projective completion intersects the hyperplane at infinity transversely.
If $\a:\pi_1(M_X)\to \G$ is an admissible epimorphism, then the $L^2$--Betti numbers of the complement $M_X$ are computed by:
\[ b_i^{(2)}(M_X,\a)= \left\{ \ba{rl} 0, & \mbox{ for }i\ne n, \\ (-1)^{n}\chi(M_X),
&\mbox{ for }i=n, \ea \right.\]
where $\chi(M_X)$ denotes the Euler characteristic of $M_X$. 

In particular,  $$(-1)^n \cdot \chi(M_X) \geq 0.$$
\end{theorem}

As an immediate consequence we note the following:
\begin{corollary} If $X\subset \C^{n}$  is a reduced affine hypersurface defined by a homogeneous polynomial (i.e., $X$ is the affine cone on a reduced projective hypersurface in $\cp^{n-1}$), and $\a:\pi_1(M_X)\to \G$ is an admissible epimorphism, then  $$b_i^{(2)}(M_X,\a)=0 \ , \ {\text for \ all} \ \ i \geq 0.$$
\end{corollary}
\noindent Indeed, it is easy to see that such an affine cone $X\subset \C^{n}$  is in general position at infinity. Moreover, by the existence of a {\it global Milnor fibration} \cite{Mi68}, the complement $M_X$ is the total space of a fibration over $\C^*$, hence $\chi(M_X)=0$.

In \cite{DJL07} it was shown that if $\scra$ is an affine hyperplane
arrangement in $\C^n$, then at most one of the $L^2$--Betti numbers
$b_i^{(2)}(\C^n\sm  \scra,\id)$  is non--zero. Theorem \ref{mainthm} can be
seen as  an analogous statement for the complement of a hypersurface in $\C^{n}$ which is in general position at infinity. In particular, we recast and generalize to arbitrary dimensions some results of \cite{LM06,FLM}, where the case $n=2$ of plane curve complements   was considered. 

\bigskip

In this note, we are also concerned with the $L^2$--Betti numbers of the infinite cyclic cover defined by the total linking number homomorphism $\phi$. More precisely, 
given an affine hypersurface $X \subset \C^n$, we denote by $\wti M_X$ the infinite
cyclic cover of $M_X$ corresponding to $\phi$. Moreover, for an admissible
epimorphism $\a:\pi_1(M_X)\to \G$ we let
$\ti{\G}:=\im\{\pi_1(\ti M_X)\to \pi_1(M_X)\xrightarrow{\a}\G\}$, and
we denote the induced map $\pi_1(\ti M_X)\to \ti{\G}$ by $\ti{\a}$.  The $L^2$--Betti numbers we are interested in are 
\[ b_i^{(2)}(\ti M_X,\ti{\a}:\pi_1(\ti M_X)\to \ti{\G}).\]
In \cite{Ma06}, the author showed that for hypersurfaces $X \subset \C^n$  in general position at infinity  the ordinary Betti numbers $b_i(\widetilde M_X)$ of the infinite cyclic cover $\wti M_X$  are finite for all $0\leq i \leq n-1$. In this note, we prove a non-commutative generalization of this fact. More precisely, we show the following:
\begin{theorem} \label{mainthm2}
Assume that the affine hypersurface $X \subset \C^n$ is in general position at infinity. Then the $L^2$--Betti numbers $b_i^{(2)}(\widetilde M_X,\ti{\a})$ are finite for all $0\leq i \leq n-1$.
\end{theorem}

As it was already observed in several recent papers, e.g., see \cite{DL,DM07,Ma06}, hypersurfaces  in general position at infinity   behave much like weighted homogeneous hypersurfaces up to homological degree $n-1$; see Prop.\ref{wh} for a computation of $L^2$--Betti numbers of weighted homogeneous hypersurface complements. The above Theorem \ref{mainthm2} comes as a confirmation of this philosophy. 

\bigskip

Another motivation for studying $L^2$--Betti numbers of the infinite cyclic cover $\wti M_X$ comes from the fact that for appropriate choices of the group $\G$ these numbers  specialize  into several classical Alexander-type invariants of the complement. 
For instance, following work of Cochran and Harvey (cf. \cite{Co04,Ha05}), we can consider  the following homomorphism
\[ \pi_n: \pi_1(M_X)\to \pi_1(M_X)/\pi_1(M_X)_r^{(m+1)}=:\G_m,\]
where given a group $G$ we denote by $G_r^{(m)}$ the $m$--th term in the rational derived series of $G$. The group $\G_m$ is a poly-torsion-free-abelian (PTFA) group and we  define the {\it higher-order degrees} 
$\delta_{i,m}(X)$  of $X$ as the dimension of the $i$-th homology of $\ti M_X$ with coefficients in the skew field
associated to $\ti{\G}_m$. Similar invariants were defined and studied in \cite{LM06,LM07} in the case of plane curves. Moreover, as noted in \cite{FLM}, the higher-order degrees $\delta_{i,m}(X)$ can be regarded as $L^2$--Betti numbers of the infinite cyclic cover $\ti M_X$, so the results of this note characterize the Cochran-Harvey invariants as well. At this point we want to emphasize that the higher-order degrees of a space $M$, hence also the  $L^2$--Betti numbers of the infinite cyclic cover $\wti M$, may as well be infinite, since $\wti M$ is not in general a finite CW-complex. For example, for a topological space $M$ with $\pi_1(M)$ a free group with at least two generators, the higher-order degrees $\delta_{1,m}$ are infinite (cf. \cite[Ex.8.2]{Ha05}). The finiteness results of this note are rigidity properties specific to the complex algebraic setting. We should also mention that if $X$ is irreducible and in general position at infinity, then it is easy to see (cf. \cite{LM06}) that for all $0 \leq i \leq n-1$ the integer $\delta_{i,0}(X)$ equals the degree of the $i$-th Alexander polynomial of $X$, or as shown in \cite{Ma06}, the $i$-th Betti numbers of the infinite cyclic cover $\wti M_X$. In the general reducible case, Libgober pointed out a nice relationship between $\delta_{i,0}(X)$ and the {\it support} of the $i$-th universal abelian Alexander module of the complement $M_X$ (for more details, see \cite{LM06} and the references therein).

\bigskip

Our result in Theorem \ref{mainthm} is reminiscent of a similar calculation by Jost-Zuo \cite{JZ} of $L^2$--Betti numbers of a compact K\"ahler manifold of non-positive sectional curvature. This was considered in relation to an old question of Hopf whether the Euler characteristic of a compact manifold $M$ of even real dimension $2n$ has sign equal to $(-1)^n$, provided $M$ admits a metric of negative sectional curvature. However, the statement of our Theorem \ref{mainthm} is metric independent (and the degree of $X$ is arbitrary). 
Finally, one should not be misled by these calculations into thinking that the $L^2$--Betti numbers of finite CW-complexes are always integers, or that most of them usually vanish. In fact, the Atiyah conjecture asserts that these $L^2$--Betti numbers are always rational; see \cite{Lu,Lu02}  for more details on this conjecture and related matters.

\bigskip

This paper is organized as follows. In Section \ref{section:l2betti} we recall the definition of $L^2$--Betti numbers and list some of their properties. We also express the Cochran-Harvey  higher-order degrees of an affine hypersurface complement as $L^2$--Betti numbers of the  infinite cyclic cover of the complement. The main results, Theorem \ref{mainthm} and \ref{mainthm2} are proved in Section \ref{van}.

\ack{The author would like to thank Anatoly Libgober and Stefan Friedl for sharing their comments on an earlier version of the paper.}

\section{$L^2$--Betti numbers} \label{section:l2betti}

\subsection{The von Neumann algebra and its localizations}

Let $\G$ be a countable group. Define $$l^2(\G):=\{ f:\G\to \C \, | \, \sum_{g\in \G}  |f(g)|^2<\infty \}$$
the Hilbert space of square-summable functions on $\Gamma$. Then $\G$ acts on $l^2(\G)$ by right multiplication, i.e., $(g\cdot f)(h)=f(hg)$.
This defines an injective map $\C[\G]\to \BB(l^2(\G))$, where $\BB(l^2(\G))$ is the set of bounded operators on $l^2(\G)$.
We henceforth view $\C[\G]$ as a subset of $\BB(l^2(\G))$.

The \emph{von Neumann algebra} $\ng$ of $\G$ is defined as the closure of $\C[\G]\subset \BB(l^2(\G))$ with respect to pointwise convergence in
$\BB(l^2(\G))$. Note that any $\ng$--module $\MC$ has a dimension $\dim_{\ng}(\MC)\in \R_{\geq 0}\cup \{\infty\}$. We refer to \cite[Def.6.20]{Lu02} for details.

\subsection{$L^2$--Betti numbers. Definition. Properties}

Let $M$ be a topological space (not necessarily compact) and let
$\a:\pi_1(M)\to \G$ be an epimorphism to a group. Denote by $M_{\G}$ the
regular covering of $M$ corresponding to $\a$, and consider 
the $\ng$--chain complex
\[ C_*^{sing}(M_{\G})\otimes_{\Z[\G]}\ng,\]
where $C_*^{sing}(M_{\G})$ is the singular chain complex of $M_{\G}$
with right $\G$--action given by covering
translation, and $\G$
acts canonically on $\ng$ on the left. The {\it $i$-th $L^2$--Betti
number of the pair $(M,\a)$} is defined as
\[ b_i^{(2)}(M,\a):=\dim_{\ng}(H_i(C_*^{sing}(M_{\G})\otimes_{\Z[\G]}\ng))\in [0,\infty].\]
We refer to \cite[Def.6.50]{Lu02} for more details. Note that if $M$ is a CW-complex of finite type, then the cellular chain complex $C_*(M_{\G})$ can be used in the above definition of $L^2$--Betti numbers.

In the following lemma we summarize some of the properties of $L^2$--Betti numbers.
We refer to  \cite[Thm.6.54,~Lem.6.53~and~Thm.1.35]{Lu02}  for the proofs.

\begin{lemma} \label{lem:propb2}
Let $M$ be a topological space  and let $\a:\pi_1(M)\to \G$ be an epimorphism to a group.
\bn
\item $b_i^{(2)}(M,\a)$ is a homotopy invariant of the pair $(M,\a)$.
\item $b_0^{(2)}(M,\a)=0$ if $\G$ is infinite and $b_0^{(2)}(M,\a)=\frac{1}{|\G|}$ if $\G$ is finite.
\item If $M$ is a finite CW--complex, then
\[ \sum_i \,(-1)^i \,b_i^{(2)}(M,\a)=\chi(M),\] where $\chi(M)$ denotes the Euler characteristic of $M$.
\en
\end{lemma}

\begin{remark} For the definition of $L^2$--Betti numbers of a pair $(M,\a)$ we do not need to require  that the homomorphism $\a$ is surjective. However, we can reduce ourselves to this case since, for an arbitrary homomorphism $\a:\pi_1(M) \to \G$, we have that \be b_i^{(2)}(M,\a:\pi_1(M)\to \im(\a))=b_i^{(2)}(M,\a:\pi_1(M)\to \G).\ee
\end{remark}

A {\it free $\Gamma$-CW complex} $\ti M$ is the same as a regular covering $p:\wti M \to M$ of a CW complex $M$ with $\G$ as group of covering transformations. As a generalization of the homotopy invariance of $L^2$--Betti numbers, we have the following result (see  \cite[Thm.1.35(1)]{Lu02}):

\begin{lemma}\label{Lem2} Let $\ti f:\wti N \to \wti M$ be a $\G$-map of free  $\Gamma$-CW complexes of finite type, and denote by $f: N \to M$ the induced map on the corresponding orbit spaces. If the homomorphism $H_i(\ti f;\C):H_i(\wti N;\C) \to H_i(\wti M;\C)$ is bijective for $i \leq d-1$ and surjective for $i=d$, then:
\be b_i^{(2)}(M;\a:\pi_1(M) \to \G) = b_i^{(2)}(N;\a \circ f_*:\pi_1(N) \overset{f_*}{\to} \pi_1(M) \overset{\a}{\to} \G) \ , \ {\rm for} \ i<d,\ee
and \be b_d^{(2)}(M;\a) \leq b_d^{(2)}(N;\a \circ f_*).\ee
\end{lemma}

Finally, the $L^2$--Betti numbers provide obstructions for a closed manifold to fiber over the circle $S^1$. More precisely, by \cite[Thm.1.39]{Lu02}, we have the following:

\begin{lemma}\label{Lem3} 
Let $M$ be a CW complex of finite type, and $f:M\to S^1$ a fibration with connected fiber $F$. Assume that the epimorphism $f_*:\pi_1(M) \to \pi_1(S^1)=\Z$ admits a  factorization $\pi_1(M)\overset{\a}{\to} \G \overset{\beta}{\to} \Z$, with $\a$ and $\beta$ epimorphisms . Then $$b_i^{(2)}(M,\a)=0 \ , \ \ {\rm for \ all} \ i \geq 0.$$
\end{lemma}

\subsection{$L^2$--Betti numbers and Cochran--Harvey invariants}\label{CH}

A group $\G$ is called \emph{locally indicable} if for
every finitely generated non--trivial subgroup $H\subset \G$ there
exists an epimorphism $H\to \Z$.  In the following we refer to a locally indicable
torsion--free amenable group as a LITFA group.  We refer to  \cite[p.256]{Lu02} for the
definition of an {\it amenable} group, but  we note that any solvable group is
amenable, while groups containing the free groups on two generators are not amenable. Also, note that a subgroup of a LITFA group is itself a LITFA group.

Denote by $S$ the set of non--zero divisors
of the ring $\ng$. By \cite[Prop.2.8]{Re98}
(see also \cite[Thm.8.22]{Lu02}) the pair $(\ng,S)$
satisfies the right Ore condition.
The ring $\ug:=\ng S^{-1}$ is called the
\emph{algebra of operators affiliated to $\ng$}. For any
$\ug$--module $\MC$ we also have a dimension $\dim_{\ug}(\MC)$. By
\cite[Thm.8.31]{Lu02} we have
\[ b_i^{(2)}(M,\a)=\dim_{\ug}(H_i(C_*^{sing}(M_{\G})\otimes_{\Z[\G]}\ug)).\]
We collect below some properties of LITFA groups (see \cite[Thm.2.2]{FLM} and the references therein):

\begin{theorem}\label{thm:tfs}
Let $\G$ be a LITFA group.
\bn
\item All non--zero elements in $\Z[\G]$ are non--zero divisors in $\ng$.
\item  $\Z[\G]$  is an Ore domain and embeds in its classical right ring of
quotients $\K(\G)$, a skew-field.
\item $\kg$ is flat over $\Z[\G]$.
\item There exists a monomorphism $\kg\to \ug$ which makes the following diagram commute
\[ \xymatrix{\Z[\G] \ar[r] \ar[dr] & \kg\ar[d] \\ &\ug.}\]
\en
\end{theorem}

\begin{remark} Since for a LITFA group $\G$, $\K(\G)$ is a skew-field, it follows that any $\K(\G)$-module is free. In particular, $\ug$ is flat as a $\kg$-module.
\end{remark}

The following result of \cite{FLM} relates $L^2$--Betti numbers to ranks of modules over skew fields.

\begin{proposition}(\cite[Prop.2.3]{FLM}) \label{prop:b2kg}
Let $\a:\pi_1(M)\to \G$ be an epimorphism to a
LITFA group $\G$. Then 
\[  b_i^{(2)}(M,\a)=\dim_{\K(\G)}(H_i(M;\K(\G))). \]
\end{proposition}

A group $\G$ is called {\it poly--torsion--free--abelian} (PTFA) if there exists a normal series
\[ 1=\G_0 \subset \G_1\subset \dots \subset \G_{n-1}\subset \G_n=\G \]
such that $\G_{i}/\G_{i-1}$ is torsion free abelian.
It is easy to see that  PTFA groups are LITFA. Note that the quotients $\pi/\pi_r^{(k)}$ of a group by terms in the rational derived series are PTFA (cf. \cite{Co04,Ha05}).

We now recall the definition of the Cochran--Harvey invariants (which in the context of complex algebraic geometry were first studied in \cite{LM06}, for plane curve complements).  Let $X$ be a hypersurface in $\C^{n}$, with complement $M_X:=\C^{n} \setminus X$.
Furthermore let
$\a:\pi_1(M_X)\to \G$ be an {\it admissible} epimorphism to a LITFA group. Recall that ``admissible" means that
 there exists a map $\wti\phi:\G\to \Z$ such that the following diagram commutes
\[ \xymatrix{ \pi_1(M_X)\ar[rr]^{\a}\ar[dr]^\phi && \G\ar[dl]_{\wti\phi}\\
 &\Z&}\]
 where $\phi:\pi_1(M_X)\to \Z$ is the total linking homomorphism.
Denote by $\widetilde M_X$ the infinite cyclic cover of $M_X$ defined by the kernel of the total linking number homomorphism $\phi$. Let  $\wti{\G}$ be the kernel of $\ti\phi:\G\to \Z$ and denote the induced homomorphism $\pi_1(\widetilde M_X)\to \ti{\G}$ by $\ti{\a}$.

Now consider the homomorphism $\pi_1(M_X)\to
\pi_1(M_X)/\pi_1(M_X)_r^{(m+1)}=:\G_m$. It is easy to see that this
homomorphism is admissible. As in \cite{LM06} we now define
\[ \delta_{i,m}(X)=\dim_{\K(\ti{\G}_m)}(H_i(\widetilde M_X;\K(\ti{\G}_m))).\]
The following result, which is an immediate corollary to Prop.\ref{prop:b2kg},   shows that the $L^2$--Betti numbers of $\widetilde M_X$ can be viewed as a generalization of the Cochran--Harvey invariants of affine hypersurface complements.

\begin{theorem} \label{thm:dimb2}
Let $X \subset \C^{n}$ be an affine hypersurface with complement $M_X$, and let $\a:\pi_1(M_X)\to \G$ be an admissible epimorphism to a LITFA group. Then, in the above notations, we have
\[ \dim_{\K(\ti{\G})}(H_i(\widetilde M_X;\K(\ti{\G})))=b_i^{(2)}(\widetilde M_X,\ti{\a}:\pi_1(\widetilde M_X)\to \ti{\G}).\]
\end{theorem}

\section{Vanishing of $L^2$-Betti numbers of hypersurface complements}\label{van}

Let $X$ be a reduced hypersurface in $\C^{n}$ ($n \geq 2$), defined by the equation
$f=f_1\cdots f_s =0$, where $f_i$ are the irreducible factors of
$f$, and let $X_i=\{f_i=0\}$ denote the irreducible components of
$X$. Embed $\C^{n}$ in $\C\PP^{n}$ by adding the hyperplane at infinity, $H$,
and let $\bar{X}$ be the projective hypersurface in $\C\PP^{n}$ defined by
the homogenization $f^h$ of $f$. 
Let $M_X$ denote the affine hypersurface complement 
$$M_X:=\C^n \setminus X.$$
Alternatively, $M_X$ can be regarded as the complement in $\cp^n$ of the
divisor $\bar X \cup H$. Then $H_1 (M_X)$ is
free abelian,  generated by the meridian loops $\gamma_i$ about the
non-singular part of each irreducible component ${X}_i$, for
$i=1,\cdots, s$ (cf. \cite{Di92}, (4.1.3), (4.1.4)).

Since $M_X$ is a $n$-dimensional affine variety, it has the
homotopy type of a finite CW-complex of dimension $n$ (e.g., see  (cf. \cite{Di92}, (1.6.7), (1.6.8)). Hence \be b_i^{(2)}(M_X,{\a})=0  \ ,\ \ {\rm for \ all} \ \  i>n.\ee

Let us now recall our notations. We start with an admissible epimorphism 
${\a}:\pi_1(M_X)\to {\G}$ to a group $\G$, and consider the induced epimorphism  $\ti{\a}:\pi_1(\widetilde M_X)\to \ti{\G}$, where $\ti{\G}:=\ker (\ti\phi:\G \to \Z)$, and $\widetilde M_X$ is the infinite cyclic cover of $M_X$ defined by the total linking number homomorphism $\phi$. 
As already noted in the introduction, $\phi$ coincides with the homomorphism $\pi_1(M_X) \to \pi_1(\C^*)=\Z$ induced by the polynomial map $f$ (cf. \cite[p.76-77]{Di92}).
The admissibility assumption implies that  the $\Gamma$-cover of $M_X$ defined by $\a$  factors through the infinite cyclic cover $\widetilde M_X$.

We begin our investigation of $L^2$--Betti numbers of hypersurface complements with the following special case:

\begin{proposition}\label{wh}
Let $X$ be an affine hypersurface defined by a weighted homogeneous polynomial
$f:\C^{n} \to \C$. Then 
\begin{enumerate}
\item All $L^2$--Betti numbers $b_i^{(2)}(M_X,{\a})$ of the complement $M_X$ vanish.
\item All $L^2$--Betti numbers $b_i^{(2)}(\widetilde M_X,\ti{\a})$ of the infinite cyclic cover $\widetilde M_X$ are finite, and  $b_i^{(2)}(\widetilde M_X,\ti{\a})=0$ for $i\geq n$.
\end{enumerate}
\end{proposition}

\begin{proof}
Since the defining polynomial $f$ of $X$ is weighted homogeneous, there exist a global {\it Milnor fibration} (e.g., see  \cite{Mi68} or \cite{Di92}, (3.1.12)): $$F=\{f=1\}
\hookrightarrow M_X=\C^{n} \setminus X \overset{f}{\to} \C^{\ast}.$$ 
Moreover, the fiber $F$ has the homotopy type of a finite CW-complex of dimension $n-1$, and $F$ is $(n-s-2)$-connected, where $s$ is the dimension of the singular locus of the hypersurface singularity germ $(X,0)$. In particular, since $X$ is reduced, $F$ is connected.
The vanishing of $L^2$--Betti numbers $b_i^{(2)}(M_X,{\a})$ of the complement follows now from Lemma \ref{Lem3}. 

Note the Milnor fiber $F$ is homotopy equivalent to the infinite cyclic cover $\widetilde M_X$ of $M_X$ corresponding to the kernel of the total linking number homomorphism. 
It follows that $\widetilde M_X$ has the homotopy type of a {\it finite} CW complex of dimension $n-1$, so the claim about the finiteness of the $L^2$--Betti numbers of $\wti M_X$ follows readily.

\end{proof}

\begin{example} If $X$ is a {\it central} hyperplane arrangement in $\C^n$ (i.e., all hyperplanes pass through the origin), then Prop.\ref{wh} yields that all $L^2$--Betti numbers $b_i^{(2)}(M_X,{\a})$ of the complement $M_X$ vanish; this fact also follows from \cite{DJL07}.
\end{example}

Affine hypersurfaces defined by homogeneous polynomials are basic examples of {\it hypersurfaces in general position at infinity}, i.e., hypersurfaces $X \subset \C^{n}$ for which the hyperplane at infinity $H$ is transversal in a stratified sense to the projective completion $\bar X \subset \cp^{n}$.  
In \cite{Ma06}, the author showed that for affine hypersurfaces $X$  in general position at infinity  the ordinary Betti numbers $b_i(\widetilde M_X)$ of the infinite cyclic cover $\wti M_X$  are finite for all $0\leq i \leq n-1$. In this paper we give a non-commutative generalization of this fact. We begin with the following comparison result.

\begin{theorem}\label{infinity} Let $X$ be a hypersurface in $\C^{n}$, and let
$S^{\infty}$ be a $(2n-1)$-sphere in $\C^{n}$ of a sufficiently large radius
(that is, the boundary of a small tubular neighborhood in $\C\PP^{n}$
of the hyperplane $H$ at infinity). Denote by
$X^{\infty}=S^{\infty} \cap X$ the link of $X$ at infinity, and by 
 $M_X^{\infty}=S^{\infty} - X^{\infty}$ its complement in $S^{\infty}$. 
 Let $\a^{\infty}$ be the composition map $$\pi_1(M_X^{\infty})  \to \pi_1(M_X) \to \G.$$ Denote by $\wti  M_X^{\infty}$ the infinite cyclic cover of $M_X^{\infty}$ defined by the composition 
 $$\pi_1(M_X^{\infty}) \to \pi_1(M_X) \overset{\phi}{\to} \Z,$$ and let $\ti\a^{\infty}:\pi_1(\wti M_X^{\infty}) \to \ti\G$ be the induced homomorphism to $\ti\G:=\ker\{\ti \phi:\G\to \Z\}$. 
Finally, let $b_i^{(2)}(M_X^{\infty};\a^{\infty})$ and $b_i^{(2)}(\wti M_X^{\infty};\ti\a^{\infty})$ be the  $L^2$--Betti numbers of $M_X^{\infty}$ and $\wti M_X^{\infty}$, respectively.
 
Then for all $i\leq n-1$  we have the inequalities
\begin{equation}\label{inf}
b_i^{(2)}(\widetilde M_X,\ti{\a}) \leq b_i^{(2)}(\wti M_X^{\infty};\ti\a^{\infty})
\end{equation}
and
\be\label{infb}
b_i^{(2)}(M_X,{\a}) \leq b_i^{(2)}(M_X^{\infty};\a^{\infty}) ,
\ee
with equalities in (\ref{inf}) and (\ref{infb}) if $i<n-1$.
\end{theorem}

\begin{proof} First, it is clear that $\a^{\infty}$ is an admissible map.
Next, note that $M_X^{\infty}$ is homotopy equivalent to
$T(H) \setminus \bar{X} \cup H$, where $T(H)$ is the tubular neighborhood of $H$ in $\C\PP^{n}$ for which $S^{\infty}$ is the boundary.  Then a classical argument based on the Lefschetz hyperplane theorem yields that the homomorphism $\pi_i(M_X^{\infty}) \to \pi_i(M_X)$ is an isomorphism for $i < n-1$ and it is surjective for $i=n-1$; see \cite{DL}[Section 4.1] for more details.
In particular, $\a^{\infty}$ is an epimorphism, as is the composite homomorphism $\pi_1(M_X^{\infty}) \to \Z$.

From the above considerations, it follows that \be\label{eq}\pi_i(M_X,M_X^{\infty})=0 \ , \  {\rm for \ all} \  \ i \leq n-1,\ee
hence $M_X$ has the homotopy type of a complex obtained from $M_X^{\infty}$ by adding cells of dimension $\geq n$. So the same is true for any covering, and in particular for the corresponding $\Gamma$-coverings. 
So  the group homomorphisms
\be\label{gh}H_i((M_X^{\infty})_{\Gamma};\Z) \to H_i((M_X)_{\Gamma};\Z)\ee
are isomorphisms if $i < n-1$ and surjective for $i=n-1$. Since these homorphisms are induced by an embedding map, they are in fact homomorphisms of $\Z\Gamma$-modules. The (in)equalities in (\ref{infb}) follow now from Lemma \ref{Lem2}.

Next note that the $\G$-cover $(M_X)_{\Gamma}$ of $M_X$ is a $\ti\G$-cover of the infinite cyclic cover $\ti M_X$. Similar considerations apply to the covers of $M_X^{\infty}$. So (\ref{gh}) can be restated as saying that  the group homomorphisms
\be\label{gh2}H_i((\ti M_X^{\infty})_{\ti\Gamma};\Z) \to H_i((\ti M_X)_{\ti\Gamma};\Z)\ee
are isomorphisms if $i < n-1$ and surjective for $i=n-1$. And, as before, these are in fact homomorphisms of $\Z\ti\G$-modules. Another application of Lemma \ref{Lem2} yields the (in)equalities of (\ref{inf}).

\end{proof}

In the next result, we restrict our attention to the case of hypersurfaces in general position at infinity. As it was already observed in a sequence of papers, e.g., see \cite{DL,DM07,Ma06},  such hypersurfaces behave much like weighted homogeneous hypersurfaces up to homological degree $n-1$.

\begin{theorem}\label{main} Assume that the affine hypersurface $X \subset \C^n$ is in general position at infinity, i.e., the hyperplane at infinity $H$ is transversal in the stratified sense to the projective completion $\bar X$. Then 
\begin{enumerate} 
\item The $L^2$--Betti numbers $b_i^{(2)}(\widetilde M_X,\ti{\a})$ of the infinite cyclic cover $\wti M_X$ are finite for all $0\leq i \leq n-1$. In particular, the Cochran-Harvey higher-order degrees $\delta_{i,m}(X)$ are finite for $0\leq i \leq n-1$ and all integers $m$.
\item The $L^2$--Betti numbers of the complement $M_X$ are computed by
\[ b_i^{(2)}(M_X,\a)= \left\{ \ba{rl} 0, & \mbox{ for }i\ne n, \\ (-1)^{n}\chi(M_X),
&\mbox{ for }i=n. \ea \right.\] 
In particular,  $$(-1)^n \cdot \chi(M_X) \geq 0.$$
\end{enumerate}
\end{theorem}

\begin{proof} For the first part of the theorem, by Thm.\ref{infinity} it suffices 
to show that the $L^2$--Betti numbers $b_i^{(2)}(\wti M_X^{\infty};\ti\a^{\infty})$ are finite for all $0\leq i \leq n-1$. (This was proved in \cite{FLM} for $n=2$.)
Note that since $\bar X$ is transversal to $H$, the space $M_X^{\infty}$ is a circle fibration
over $H \setminus  \bar{X} \cap H$ which is homotopy equivalent to the complement in $\C^{n}$ to the
affine cone over the projective hypersurface $\bar{X} \cap H \subset H=\mathbb{CP}^{n-1}$ (for a similar argument see \cite{DL}, Section 4.1). Let $\{h=0\}$ be the polynomial defining $\bar{X} \cap H$ in $H$. Then the infinite cyclic cover $\wti M_X^{\infty}$ of $M_X^{\infty}$ is
homotopy equivalent to the Milnor fiber $\{h=1\}$ of the (homogeneous) hypersurface
singularity at the origin defined by $h$. In particular,  $\wti M_X^{\infty}$ has the homotopy type of a finite CW-complex. So the claim about the finiteness of $b_i^{(2)}(\wti M_X^{\infty};\ti\a^{\infty})$ follows now from definition. Similarly, the finiteness of the Cochran-Harvey higher-order degrees $\delta_{i,m}(X)$ in the relevant range follows from the considerations of Section \ref{CH}, where these degrees are realized as $L^2$--Betti numbers of the infinite cyclic cover $\wti M_X$.

For the second part of the theorem, note that by the above considerations $M_X^\infty$ is homotopic to the total space of a fibration over $S^1$, namely the Milnor fibration at the origin corresponding to the homogeneous polynomial $h$. So, by Lemma \ref{Lem3} we obtain  that all $L^2$--Betti numbers $b_i^{(2)}(M_X^{\infty};\a^{\infty})$ vanish. The claim about the $L^2$--Betti numbers of $M_X$ now follows from the inequalities (\ref{infb}) of Thm.\ref{infinity}, together with Lem.\ref{lem:propb2}(3).

\end{proof}

\begin{remark} If $\G$ is a LITFA group as in Section \ref{CH}, then the $L^2$--Betti numbers are determined by ranks of homology modules over skew-fields. In this case, the flatness of certain rings involved shows that the finiteness of the $L^2$--Betti number $b_i^{(2)}(\widetilde M_X,\ti{\a})$ of the infinite cyclic cover is equivalent to the vanishing of the $L^2$--Betti number $b_i^{(2)}(M_X,\a)$ of the complement.
\end{remark}

\end{document}